\documentclass{amsart}

\usepackage[active]{srcltx}

\usepackage{amsmath,amssymb,amsthm,amscd}
\newtheorem{theorem}{Theorem}[section]

\newtheorem{proposition}[theorem]{Proposition}
\newtheorem{corollary}[theorem]{Corollary}

\begin{document}

\title[when cotorsion modules are pure-injective]{Commutative rings whose cotorsion modules are pure-injective}
\author{Fran\c{c}ois Couchot}

\address{Universit\'e de Caen Basse-Normandie, CNRS UMR
  6139 LMNO,
F-14032 Caen, France}
\email{francois.couchot@unicaen.fr}  



\begin{abstract}
Let $R$ be a ring (not necessarily commutative). A left $R$-module is said to be {\it cotorsion} if $\mathrm{Ext}_R^1(G,M)=0$ for any flat $R$-module $G$. It is well known that each pure-injective left $R$-module is cotorsion, but the converse does not hold: for instance, if $R$ is left perfect but not left pure-semisimple then, each left $R$-module is cotorsion but there exist non-pure-injective left modules. The aim of this paper is to describe the class $\mathcal{C}$ of commutative rings $R$ for which each cotorsion $R$-module is pure-injective. It is easy to see that $\mathcal{C}$ contains the class of von Neumann regular rings and the one of pure-semisimple rings. We  prove that $\mathcal{C}$ is strictly contained in the class of locally pure-semisimple rings. We  state that a commutative ring $R$ belongs to $\mathcal{C}$ if and only if $R$ verifies one of the following conditions:
\begin{enumerate}
\item $R$ is coherent and each pure-essential extension of $R$-modules is essential;
\item $R$ is coherent and each RD-essential extension of $R$-modules is essential;
\item any $R$-module $M$ is pure-injective if and only if $\mathrm{Ext}_R^1(R/A,M)=0$ for each pure ideal $A$ of $R$ (Baer's criterion).
\end{enumerate}
\end{abstract}

\maketitle

\section{Introduction and preliminaries}

The aim of this study is to give a complete description of commutative rings for which each cotorsion module is pure-injective. In this first section we recall some definitions and some former results. Then, in section \ref{S:gc}, we enunciate and show some partial results which are available even if the ring is not commutative.  
Section \ref{S:cc} is devoted to the commutative case. We get our main result (Theorem \ref{T:main}) by using localizations and local rings. In the last section we show that a commutative ring $R$ is locally perfect if and only if any $R$-module $M$ for which $\mathrm{Ext}^1_R(C,M)=0$ for each cyclic flat module $C$ is cotorsion, and we investigate the following question: give a characterization of rings $R$ for which each flat-essential exension of $R$-modules is essential. Throughout this paper other related questions are studied, where we use the following notions: Warfield cotorsion module, RD-injective module, RD-essential extension and so on... 

Even in the commutative case some questions are open. For instance, the condition "each cotorsion module is pure-injective" implies the condition " each Warfield cotorsion module is RD-injective", but the converse is not proven. On the other hand, we do not know if there exist non-coherent commutative rings for which each pure-essential extension of modules is essential. Also, it should be interesting to study strongly perfect rings which are introduced in the last section.

We shall assume that all rings are associative with identity and
all modules are unitary.
Given a ring $R$, any left module $M$ is said to be {\bf P-flat} (resp. {\bf P-injective}) if $\mathrm{Tor}^R_1(R/rR,M)=0$ (resp. $\mathrm{Ext}_R^1(R/Rr,M)=0$) for each $r\in R$. We say that $R$ is {\bf left P-coherent} if each principal left ideal of $R$ is finitely presented.

A left module $M$ is {\bf FP-injective} if $\mathrm{Ext}^1_R(F,M)=0$ for each finitely presented left module $F$.

Any left module $M$ is called {\bf cotorsion} (resp. {\bf Warfield cotorsion}) if, for each flat (resp. P-flat) left module $F$, $\mathrm{Ext}^1_R(F,M)=0$.

A short exact sequence of left modules is {\bf pure} (resp. {\bf RD-pure}) if it remains exact when tensoring it by any right module (resp. module of the form $R/rR$, $r\in R$). A left module is {\bf pure-injective} (resp. {\bf RD-injective}) if it is injective relatively to each pure (resp. RD-pure) exact sequence of left modules.

$R$ is said to be left {\bf pure-semisimple} (resp. {\bf RD-semisimple}) if each left $R$-module is pure-injective (resp. RD-injective). When $R$ is commutative then $R$ is pure-semisimple if and only if it is RD-semisimple if and only if it is an Artinian ring whose all ideals are principal (\cite[Theorem 4.3]{Gri70}).

An $R$-module $B$
is a \textbf{pure-essential extension} (resp. {\bf RD-essential extension}) of a submodule $A$ if $A$ is a
pure (resp. RD) submodule of $B$ and, if for each
submodule $K$ of $B$, either $K\cap A\ne 0$ or $(A+K)/K$ is not a pure (resp. RD)
submodule of $B/K$. 
We say that $B$ is a \textbf{pure-injective hull} (resp. {\bf RD-injective hull})
of $A$ if $B$ is pure-injective (resp. RD-injective) and a pure-essential (resp. RD-essential) extension of $A$.

 Each $R$-module $M$ has a pure-injective hull and an RD-injective hull (\cite[Proposition 6]{War69}).

A left module $B$ is a {\bf flat extension} (resp. {\bf P-flat extension}) of a submodule $A$ if $B/A$ is flat (resp. P-flat). Moreover, if there are no submodules $S$ of $B$ with $S\cap A=0$ and $B/S$ flat (resp. P-flat) extension of $A$, then $B$ is a {\bf flat essential extension} (resp. {\bf P-flat essential extension}) of $A$. If $B$ is cotorsion (resp. Warfield cotorsion) and a flat (resp. P-flat) essential extension of a submodule $A$ then we say that $B$ is a {\bf cotorsion} (resp. {\bf Warfield cotorsion}) {\bf envelope} of $A$ (by \cite[Theorem 3.4.5]{Xu96} these definitions are equivalent to the usual ones). 

Each left module $M$ has a cotorsion (resp. Warfield cotorsion) envelope (\cite[Theorem 6]{BiEBEn01} and \cite[Theorem 3.4.6]{Xu96}).

For each left module $M$ we denote by $\mathcal{E}(M)$ its cotorsion envelope, $\mathcal{E}_W(M)$ its Warfield cotorsion envelope, $\mathrm{PE}(M)$ its pure-injective hull and $\mathrm{RDE}(M)$ its RD-injective hull. 

Each pure(RD)-injective module is (Warfield) cotorsion, but \cite[Example p.75]{Xu96} shows that the converse does not hold.

\begin{theorem}\cite[Theorem 3.5.1]{Xu96}\label{T:Xu}
For any ring $R$ the following are equivalent:
\begin{enumerate}
\item for any exact sequence of left modules $0\rightarrow G'\rightarrow G\rightarrow G''\rightarrow 0$ with $G'$ and $G''$ pure-injective, $G$ is also pure-injective;
\item for any left module $M$, $\mathrm{PE}(M)/M$ is flat;
\item every cotorsion left module is pure-injective.
\end{enumerate}
Moreover if $R$ is right coherent, then the above are equivalent to the following:
\begin{enumerate}
\item[(4)] for any exact sequence of left modules $0\rightarrow G'\rightarrow G\rightarrow G''\rightarrow 0$ with $G'$ and $G$ pure-injective, $G''$ is also pure-injective.
\end{enumerate}
\end{theorem}

By a similar way the following can be proven.

\begin{theorem}
For any ring $R$ the following are equivalent:
\begin{enumerate}
\item for any exact sequence of left modules $0\rightarrow G'\rightarrow G\rightarrow G''\rightarrow 0$ with $G'$ and $G''$ RD-injective, $G$ is also RD-injective;
\item for any left module $M$, $\mathrm{RDE}(M)/M$ is P-flat;
\item every Warfield cotorsion left module is RD-injective.
\end{enumerate}
Moreover if $R$ is right P-coherent, then the above are equivalent to the following:
\begin{enumerate}
\item[(4)] for any exact sequence of left modules $0\rightarrow G'\rightarrow G\rightarrow G''\rightarrow 0$ with $G'$ and $G$ RD-injective, $G''$ is also RD-injective.
\end{enumerate}
\end{theorem}

The following proposition is well known. For convenience, a proof is given. We set  $0_P$ to be the kernel of the natural map $R\rightarrow R_P$ where $P\in\mathrm{Spec}\ R$.

\begin{proposition}
\label{P:pireg}
Let $R$ be a commutative ring. We assume that each prime ideal is maximal. Then:
\begin{enumerate}
\item for any closed subset $C$ of $\mathrm{Spec}\ R$, $C=V(A)$ where $A=\cap_{P\in C}0_P$ is a pure ideal;
\item for each maximal ideal $P$, $R_P=R/0_P$;
\item each pure ideal of $R$ is generated by idempotents.
\end{enumerate}
\end{proposition}
\begin{proof}
$(1)$. Let $C=V(B)$ where $B=\cap_{L\in C}L$. We put $A=\cap_{P\in C}0_P$. Let $b\in B$ and $P\in C$. The image of $b$, by the natural map $R\rightarrow R_P$, belongs to the nilradical of $R_P$. It follows that there exist  $0\ne n_P\in\mathbb{N}$ and $s_P\in R\setminus P$ such that $s_Pb^{n_P}=0$. Hence, $\forall L\in D(s_P)\cap C,\ b^{n_P}\in 0_L$. A finite family $(D(s_{P_j}))_{1\leq j\leq m}$ covers $C$. Let $n=\max\{n_{P_1},\dots,n_{P_m}\}$. Then $b^n\in 0_L,\ \forall L\in C$, whence $b^n\in A$. We deduce that $C=V(A)$. Now, we have $A_P=0$ if $P\in V(A)$ and $A_P=R_P$ if $P\in D(A)$. Hence $A$ is a pure ideal.

$(2)$ is a consequence of $(1)$ by taking $C=\{P\}$.

$(3)$. We know that $\mathrm{Spec}\ R$ is homeomorphic to $\mathrm{Spec}\ R/N$ where $N$ is the nilradical of $R$. Since $R/N$ is von Neumann regular its principal ideals are generated by idempotents. So, $\mathrm{Spec}\ R$ has a base of clopen subsets (closed and open). Whence if $A$ is a pure ideal then, for any $a\in A$ there exists an idempotent $e_a$ such that $D(e_a)=D(a)\subseteq D(A)$. Clearly $D(A)=D(\Sigma_{a\in A}Re_a)$. Since $\Sigma_{a\in A}Re_a$ is a pure ideal, by $(1)$ we conclude that $A=\Sigma_{a\in A}Re_a$.
\end{proof}

\section{when cotorsion modules are pure-injective: general case}
\label{S:gc}
A left module $M$ over a ring $R$ is called {\bf regular} (respectively {\bf RD-regular}) if all its submodules are pure (respectively RD-pure).

\begin{proposition}\label{P:reg}
\label{P:regular} Let $R$ be a ring, $J$ its Jacobson radical. Let $M$ be an RD-regular left $R$-module. Then $JM=0$ and,  if in addition $R$ is semilocal, $M$ is semisimple.
\end{proposition}
\begin{proof}
 If $0\ne x\in M$ then $Rax$ is an RD-submodule of $Rx$ for each $a\in R$. So, for each $a\in J$, there exists $b\in R$ such that $ax=abax$. It follows that $(1-ab)ax=0$, and from $a\in J$ we successively deduce that $(1-ab)$ is a unit and $ax=0$.
\end{proof}

\begin{proposition}\label{P:noabelien}
Let $R$ be a ring. Assume there exists a family $\mathfrak{E}$ of orthogonal central idempotents of $R$ satisfying the following conditions:
\begin{itemize}
\item[(a)] $R/R(1-e)$ is a left pure-semisimple ring , for each $e\in\mathfrak{E}$;
\item[(b)] $R/A$ is a von Neumann regular ring where $A=\oplus_{e\in\mathfrak{E}}Re$.
\end{itemize}
Then:
\begin{enumerate}
\item each cotorsion left $R$-module is pure-injective;
\item  each pure-essential extension of left $R$-modules is essential;
\item any left $R$-module $M$ is pure-injective if and only if $\mathrm{Ext}_R^1(C,M)=0$ for each cyclic flat left $R$-module $C$;
\item for any left $R$-module $M$, $\mathrm{PE}(M)/M$ is flat, FP-injective and regular.
\end{enumerate}
\end{proposition}
\begin{proof}
$(3)$ and $(1)$. Let $M$ be a left $R$-module satisfying $\mathrm{Ext}_R^1(R/B,M)=0$ for each pure left ideal $B$ of $R$. Since $A$ is a pure ideal, the following sequence is exact:
\[0\rightarrow\mathrm{Hom}_R(R/A,M)\rightarrow\mathrm{Hom}_R(R,M)\rightarrow\mathrm{Hom}_R(A,M)\rightarrow 0.\]
Let $C$ be a left ideal of $R/A$. Since $R/A$ is von Neumann regular, $C$ is a pure ideal and its inverse image $B$ by the natural map $R\rightarrow R/A$ is a pure left ideal of $R$. From $\mathrm{Ext}_{R/A}^1(R/B,\mathrm{Hom}_R(R/A,M))\cong\mathrm{Ext}_R^1(R/B,M)=0$ we deduce that $\mathrm{Hom}_R(R/A,M)$ is injective over $R/A$ and $R$. So, the above sequence splits. On the other hand $\mathrm{Hom}_R(A,M)\cong\prod_{e\in\mathfrak{E}}eM$. Since $R/R(1-e)$ is left pure-semisimple, it successively follows that $eM$ is pure-injective for each $e\in\mathfrak{E}$, $\mathrm{Hom}_R(A,M)$ is pure-injective and $M$ too.

$(2)$. Let $M$ be a left $R$-module, $N=\mathrm{Hom}_R(R/A,M)$, $E=\mathrm{E}(N)$ and $L=\mathrm{Hom}_R(A,M)$. As above $L$ is pure-injective. So, $E\oplus L$ is pure injective. The inclusion map $N\rightarrow E$ extends to a homomorphism $f:M\rightarrow E$. Let $g:M\rightarrow L$ be the canonical map and $L'$ its image. Then, it is easy to check that the homomorphism $\phi:M\rightarrow E\oplus L$ defined by $\phi(m)=(f(m),g(m))$ for each $m\in M$ is injective. Since $R/A$ is von Neumann regular,  $E/N$ is flat. It is easy to see that $AL=AL'$. So, $L/L'$ is also an $R/A$-module. It follows that $\mathrm{coker}(\phi)$ is an $R/A$-module which is flat over $R$. Hence $\phi$ is a pure monomorphism. Let $(x,y)\in E\oplus L$. First assume that $y\ne 0$. There exists $e\in\mathfrak{E}$ such that $ey\ne 0$. So, there exists $z\in M$ such that $ey=g(z)$. It follows that $ey=g(ez)$ and $\phi(ez)=e(x,y)=(0,ey)$. If $y=0$ then there exists $s\in R$ such that $0\ne sx\in N$, whence $\phi(sx)=s(x,0)$. Hence $\phi$ is an essential monomorphism.

$(4)$. Since $\mathrm{coker}(\phi)$ is a module over $R/A$ which is a von Neumann regular ring and flat as right $R$-module, we have $\mathrm{coker}(\phi)$ is  flat, FP-injective and regular as $R$-module.
\end{proof}

\begin{proposition}\label{P:noabelien1}
Let $R$ be a ring. Assume there exists a family $\mathfrak{E}$ of orthogonal central idempotents of $R$ satisfying the following conditions:
\begin{itemize}
\item[(a)]  $R/R(1-e)$ is a left RD-semisimple ring , for each $e\in\mathfrak{E}$;
\item[(b)] $R/A$ is a von Neumann regular ring where $A=\oplus_{e\in\mathfrak{E}}Re$.
\end{itemize}
Then:
\begin{enumerate}
\item each Warfield cotorsion left $R$-module is RD-injective;
\item each RD-essential extension of left $R$-modules is essential;
\item any left $R$-module $M$ is RD-injective if and only if $\mathrm{Ext}_R^1(C,M)=0$ for each cyclic flat left $R$-module $C$.
\item for any left $R$-module $M$, $\mathrm{RDE}(M)/M$ is flat, FP-injective and regular.
\end{enumerate}
\end{proposition}

As in \cite{Mat59}  a left $R$-module $M$ is said
to be {\bf semi-compact} if every finitely solvable set of congruences $x \equiv x_{\alpha}$ (mod $M[I_{\alpha}]$)
 (where $\alpha\in \Lambda$,  $x_{\alpha}\in M$ and  $I_{\alpha}$ is a left ideal of $R$ for each $\alpha \in \Lambda$) has a simultaneous solution in $M$.

\begin{proposition}
\label{P:semi-compact} Let $R$ be a ring. Assume that each pure-essential extension of left $R$-modules is essential. Then each semi-compact left module is pure-injective.
\end{proposition}
\begin{proof}
Let $M$ be a semi-compact left $R$-module. By way of contradiction assume there exists $x\in\mathrm{PE}(M)\setminus M$. Let $A=\{a\in R\mid ax\in M\}$. Then $A\ne 0$ since the extension $M\rightarrow\mathrm{PE}(M)$ is essential. We consider the following system of equations: $aX=ax,\ a\in A$. Since $M$ is a pure submodule, for each finite subset $B$ of $A$, there exists $x_B\in M$ such that $ax_B=ax$ for each $a\in B$. By \cite[Proposition 1.2]{BeCoSh14} the semi-compactness of $M$ implies that  there exists $y\in M$ such that $ax=ay$ for each $a\in A$. It follows that $R(x-y)\cap M=0$ which contradicts that $M$ is essential in $\mathrm{PE}(M)$.
\end{proof}

\begin{proposition}\label{P:eregular}
Let $R$ be a  ring. Assume that each pure-essential extension of left modules is essential. Then, for each  FP-injective left $R$-module $M$, $\mathrm{PE}(M)/M$ is  regular.
\end{proposition}
\begin{proof}
Since $M$ is FP-injective, we have $\mathrm{PE}(M)$ is an injective hull of $M$. Let $C$ be a submodule of $\mathrm{PE}(M)/M$ and $A$ its inverse image by the natural epimorphism $\mathrm{PE}(M)\rightarrow\mathrm{PE}(M)/M$. The inclusion map $M\rightarrow A\rightarrow\mathrm{PE}(A)$ is an essential extension. Hence $\mathrm{PE}(A)\cong\mathrm{PE}(M)$. Then $\mathrm{PE}(A)$ is injective and $A$ is FP-injective, whence $A$ is a pure submodule of $\mathrm{PE}(M)$ and $C$ a pure submodule of $\mathrm{PE}(M)/M$. 
\end{proof}

In the same way and by using Proposition \ref{P:reg} we get the following.

\begin{proposition}\label{P:RD}
Let $R$ be a  ring. Assume that each RD-essential extension of left $R$-modules is essential. Then, for each  P-injective left $R$-module $M$, $\mathrm{RDE}(M)/M$ is RD-regular.
\end{proposition}

\begin{proposition}
\label{P:quot} Let $R$ be a ring. Assume that each RD-essential extension of left modules is essential. Then, for any two-sided ideal $A$, each RD-essential extension of left $R/A$-modules is essential.
\end{proposition}
\begin{proof} Let $\alpha:M\rightarrow N$ be an RD-essential extension of left $R/A$-modules and let $\beta:M\rightarrow E$ be an RD-injective hull of $M$ over $R$. Then there exists a homomorphism $\gamma:N\rightarrow E$ such that $\beta=\gamma\alpha$. From the fact that $\alpha$ is RD-essential and $\beta$ is an RD-monomorphism we deduce that $\gamma$ is injective. We conclude that $\alpha$ is essential.
\end{proof}

\section{when cotorsion modules are pure-injective: commutative case}
\label{S:cc}
\begin{proposition}
\label{P:local} Let $R$ be a commutative ring. Assume that each pure(RD)-essential extension of $R$-modules is essential. Then for each multiplicative subset $S$ of $R$, each pure(RD)-essential extension of $S^{-1}R$-modules is essential.
\end{proposition}
\begin{proof}
Let $A\rightarrow B$ be a pure-essential extension of $S^{-1}R$-modules, and let $C$ be an $R$-submodule of $B$ such that $A\cap C=0$ and $A$ is a pure submodule of $B/C$. It is easy to check that $A\cap S^{-1}C=0$ and $A$ is a pure submodule of $B/S^{-1}C$. So, $S^{-1}C=0$ and $C=0$. Then $A\rightarrow B$ is a pure-essential extension of $R$-modules. Now it is easy to conclude.
\end{proof}

Recall that a ring $R$ is left {\bf perfect} if each flat left $R$-module is projective.

\begin{proposition}
\label{P:semiCompPurInj} Let $R$ be a commutative ring. Assume that each semi-compact $R$-module is pure-injective. Then each prime ideal is maximal.
\end{proposition}
\begin{proof}
Let $L$ be a prime ideal of $R$, $R'=R/L$ and $M$ a flat $R'$-module. Since $R'$ is a domain, each flat $R'$-module is semi-compact over $R'$ and over $R$ too. It follows that each flat $R'$-module is pure-injective. There is a pure-exact sequence $0\rightarrow K\rightarrow F\rightarrow M\rightarrow 0$ where $F$ is a free $R'$-module. So, $K$ is flat and pure-injective over $R'$. We deduce that the above sequence splits and consequently $M$ is projective. Hence $R'$ is a perfect domain, whence $R'$ is a field and $L$ is maximal.
\end{proof}

\begin{theorem}
\label{T:locNoet} Let $R$ be a commutative ring satisfying each pure-essential extension of $R$-modules is essential. Then $R_P$ is pure-semisimple for any  maximal ideal $P$.
\end{theorem}
\begin{proof} By Proposition \ref{P:local} we may assume that $R$ is local of maximal ideal $P$. Let $I=\mathrm{E}_R(R/P)$, $M=I^{(\mathbb{N})}$, $E=\mathrm{E}_R(M)$ and $S=E/M$. By Propositions \ref{P:eregular} and \ref{P:reg} $S$ is semisimple. Let $0\ne a\in P$, $A=(0:a)$ and $R_a=R/A$. By Propositions \ref{P:semi-compact} and \ref{P:semiCompPurInj} $P$ is the sole prime ideal. So, $A\ne 0$. For any $R$-module $G$ we put $G'=\{g\in G\mid Ag=0\}=G[A]$. Then $I'=\mathrm{E}_{R_a}(R/P)=aI$, $M'=I'^{(\mathbb{N})}=aM$, $E'=aE$ and $E'$ is injective over $R_a$. Since $aS=0$, we have $M'=E'$. By \cite[Theorem 25.3]{AnFu92} $R_a$ is Noetherian, and Artinian since $P$ is the sole prime ideal. Let $(Ra_n)_{n\in\mathbb{N}}$ be a descending chain of proper ideals of $R$. We may assume that $a_0\in P$. If we choose $a=a_0$, then $Ra$ is an $R_a$-module. So, it is Artinian and consequently $R$ satisfies the descending condition on principal ideals. We conclude that $R$ is perfect by \cite[43.9]{Wis91}. It follows that $P^2\ne P$. By way of contradiction suppose that $P/P^2$ is a vector space of dimension $\geq 2$ over $R/P$. Then there is a Noetherian factor $R'$ of $R$ modulo a suitable ideal whose maximal ideal is generated by $2$ elements. So, $R'$ is not pure-semisimple. But, we successively get that each $R'$-module is semicompact (because $R'$ is Noetherian), and pure-injective by Proposition \ref{P:semi-compact}. From this contradiction we deduce that $P$ is principal and $R$  pure-semisimple.
\end{proof}

\begin{theorem}
\label{T:locNoetRD} Let $R$ be a commutative ring satisfying each RD-essential extension of $R$-modules is essential. Then $R_P$ is pure-semisimple for any maximal ideal $P$.
\end{theorem}
\begin{proof} By Proposition \ref{P:local} we may assume that $R$ is local of maximal ideal $P$. Let $S=R/P$ and $E$ an $R$-module containing $S$ as proper RD-submodule. Let $x\in E\setminus S$. If $ax\in S$ then there exists $s\in S$ such that $ax=as$. Since $a\in P$, we have $ax=0$ and $Rx\cap S=0$. Hence $S$ is RD-injective.
 
First assume that $R$ is Noetherian. By \cite[Corollary 4.7]{Couch06} $R$ is a chain ring. So, each RD-exact sequence is pure. Consequently $R$ satisfies the assumption of Theorem \ref{T:locNoet}, whence $R$ is pure-semisimple. 

Now, assume that $R$ is not Noetherian. Let $L$ be a prime ideal and $R'=R/L$. Suppose that $L\ne P$. Each RD-essential extension of $R'$-modules is essential by Proposition \ref{P:quot}. Then, from the first part of the proof $R'$ is not Noetherian. Let $N$ be a FP-injective $R'$-module which is not injective, $E=\mathrm{E}_{R'}(N)$ and $T=E/N$. Let $a\in P\setminus L$. Then $E=aE$, whence $T=aT$. But, by Propositions \ref{P:RD} and \ref{P:reg} $T$ is semisimple, whence $aT=0$. From this contradiction we deduce that $L=P$. Now, we do as in the proof of Theorem \ref{T:locNoet} to show that $R$ is perfect. So, $P/P^2$ is of infinite dimension over $R/P$. Whence there exists Noetherian factor rings of $R$ which are not pure-semisimple.  This contradicts the beginning of the proof. Hence $R$ is pure-semisimple.
\end{proof}

To show the following proposition we adapt the proof of \cite[Example p.75]{Xu96}.

\begin{proposition}
\label{P:primemax} Let $R$ be a commutative ring for which each cotorsion $R$-module is pure-injective. Then each prime ideal of $R$ is maximal. 
\end{proposition}
\begin{proof}
By way of contradiction suppose there exists a non-maximal prime ideal $L$. Since each $R/L$-module is pure-injective over $R/L$ if and only if it is pure-injective over $R$, each cotorsion  $R/L$-module is pure-injective by Theorem \ref{T:Xu}. So, we may assume that $R$ is an integral domain. Let $P$ be a maximal ideal of $R$ containing $a\ne 0$. We put $I=\mathrm{E}(R/P)$, $I_n=I[a^n]$ for each integer $n\geq 1$, $M=\oplus_{n\geq 1}I_n$ and $N=\prod_{n\geq 1}I_n$. Then $N$ is pure-injective and $N=N_1\oplus N_2$ where $N_1=\mathrm{PE}(M)$. By Theorem \ref{T:Xu} $N_1/M$ is torsionfree since it is flat.

Let $S=\cup_{n\geq 1}N[a^n]$. Clearly $M\subset S$. Let $x=(x_n)_{n\geq 1}\in S$ and $k$ an integer $\geq 1$. There exists an integer $p\geq 1$ such that $a^px=0$. Let $n\geq p+k$. Then $x_n=a^ky_n$ where $y_n\in I$. But $0=a^{n-k}x_n=a^ny_n$ whence $y_n\in I_n$. So the elements of $S/M$ are divisible by $a^k$ for each $k\geq 1$. Consider the projection $\pi_2: N\rightarrow N_2$ and its restriction to $S$. Since $M$ is in the kernel of $\pi_2$, there is an induced homomorphism $\bar{\pi}_2: N/M\rightarrow N_2$. Note that $N$ (and $N_2$ too) has no nonzero elements divisible by $a^k$ for all $k\geq 1$. This implies that $\bar{\pi}_2$ maps $S/M$ to zero in $N_2$. Thus $S\subseteq N_1$, so $S/M\subseteq\mathrm{PE}(M)/M$. But $S/M\ne 0$ is not torsionfree. So, we get the desired contradiction.
\end{proof}

\begin{proposition}\label{P:redu}
Let $R$ be a ring, $E$ a left $R$-module and $U$ a pure submodule of $E$. Then the following conditions are equivalent:
\begin{enumerate}
\item $E/U$ is FP-injective if $E$ is FP-injective;
\item $E/U$ is FP-injective if $E$ is an injective hull of $U$.
\end{enumerate}
\end{proposition}
\begin{proof}
It is obvious that $(1)\Rightarrow (2)$.

$(2)\Rightarrow (1)$. First we assume that $E$ is injective. Then $E$ contains a submodule $E'$ which is an injective hull of $U$. Since $E/E'$ is injective and $E'/U$ FP-injective, $E/U$ is FP-injective too. Now we assume that $E$ is FP-injective. Let $H$ be the injective hull of $E$. Then $E/U$ is a pure submodule of $H/U$. We conclude that $E/U$ is FP-injective.
\end{proof}

\begin{theorem}\label{T:max}
Let $R$ be a commutative ring. Assume that each cotorsion $R$-module is pure-injective. Then:
\begin{enumerate}
\item  for each maximal ideal $P$, $R_P$ is pure-semisimple;
\item $R$ is coherent.
\end{enumerate}
\end{theorem}
\begin{proof} 
$(1)$.  For any maximal ideal $P$, each cotorsion $R_P$-module is pure-injective over $R_P$. So, we may assume that $R$ is local and $P$ is its maximal ideal. Now we do as in the beginning of the proof of Theorem \ref{T:locNoet} with the same notations. Thus $E'=aE$ is the pure-injective hull of $M'=aM$. It follows that $aS$ is flat over $R$. Since $P$ is the sole prime ideal of $R$ by Proposition \ref{P:primemax}, $a$ is nilpotent. Let $n>0$ be the smallest integer satisfying $a^{2n}S=0$. Since $a^nS$ is flat, we have $S[a^n]$ is a pure submodule of $S$. For each $s\in S$, $a^ns\in S[a^n]$ and there exists $x\in S[a^n]$ such that $a^ns=a^nx=0$. So, $a^nS=0$. It is easy to see that necessarily $n=1$ and $aS=0$. From $PS=0$, $S$ flat and $R$ local ring, we deduce that $S=0$ if $P\ne 0$. It follows that $R$ is Artinian. Hence each $R$-module is cotorsion. We conclude that each $R$-module is pure-injective and $R$ is pure-semisimple.

$(2)$. We shall prove that $E/U$ is FP-injective for any FP-injective module $E$ and any pure submodule $U$ of $E$. By Proposition \ref{P:redu} we may assume that $E$ is the injective hull of $U$. So, $E\cong\mathcal{E}(U)$. By Theorem \ref{T:Xu} $E/U$ is flat. Then, for each maximal ideal $P$, $(E/U)_P$ is flat, hence free and injective since $R_P$ is pure-semisimple. We conclude that $E/U$ is FP-injective and $R$ is coherent by \cite[35.9]{Wis91}. 
\end{proof}

\begin{proposition}\label{P:locallyPS}
Let $R$ be a commutative ring whose prime ideals are  maximal. Let $X$ be the set of all maximal ideals $P$ such that $PR_P=0$. We denote by $A$,  the kernel of the naturel map $R\rightarrow \prod_{P\in X}R_P$. If $R$ is P-coherent then, $A$ is a pure submodule of $R$ and $X=V(A)$.
\end{proposition}
\begin{proof}
Since $R/A$ is a subring of a product of fields, $R/A$ is reduced. From the fact that each prime ideal is maximal we deduce that $R/A$ is von Neumann regular. Thus $R/A$ is a pure submodule of $\prod_{P\in X}R_P$ which is P-flat because $R$ is P-coherent. It follows that $A$ is a pure ideal. Since $A_P=0$ for each $P\in X$, we have $X\subseteq V(A)$. Let $P\in V(A)$. Then $A_P=0$ because $A$ is pure. It is obvious that $J\subseteq A$ where $J$ is the Jacobson radical of $R$. Since $J$ is also the nilradical of $R$, we have $PR_P=JR_P=0$. Hence $P\in X$.
\end{proof}

\begin{theorem}\label{T:main}
Let $R$ be a commutative ring. The following conditions are equi\-valent:
\begin{enumerate}
\item each cotorsion $R$-module is pure-injective;
\item $R$ is P-coherent and each pure-essential extension of $R$-modules is essential;
\item $R$ is P-coherent and each RD-essential extension of $R$-modules is essential;
\item any $R$-module $M$ is pure-injective if and only if $\mathrm{Ext}_R^1(C,M)=0$ for each cyclic flat $R$-module $C$;
\item there exists a family $\mathfrak{E}$ of orthogonal irreducible idempotents of $R$ satisfying the following conditions:
\begin{enumerate}
\item $R/R(1-e)$ is a pure-semisimple ring but not a field, for each $e\in\mathfrak{E}$;
\item $R/A$ is a von Neumann regular ring where $A=\oplus_{e\in\mathfrak{E}}Re$.
\end{enumerate}
\end{enumerate}
Moreover, when these conditions hold, the following are satisfied:
\begin{itemize}
\item[(6)] $\mathrm{PE}(M)/M$ is flat, FP-injective and regular for each $R$-module $M$.
\item[(7)] each Warfield cotorsion module is RD-injective.
\end{itemize}
\end{theorem}
\begin{proof}
It is obvious that $(4)\Rightarrow (1)$. If $R$ satisfies condition $(2)$
or $(3)$ then,  by Theorems  
\ref{T:locNoet} or   \ref{T:locNoetRD}, $R$ is arithmetical. It follows that $(2)\Leftrightarrow (3)$.

$(1)\Rightarrow (5)$. By Theorem \ref{T:max}, $R$ is coherent and $R_P$ is pure-semisimple for each maximal ideal $P$. Let $A$ be the pure ideal of $R$ defined in Proposition \ref{P:locallyPS}. By Proposition \ref{P:pireg} $A$ is generated by its idempotents. Let $e=e^2\in A$. Then $R'=R/R(1-e)$ satisfies $(1)$. Let $I=\oplus_{P\in D(e)}\mathrm{E}(R/P)$, $M=I^{(\mathbb{N})}$, $E=\mathrm{E}_{R'}(M)$ and $S=E/M$. For each nilpotent element $a$ of $R'$, we do as in the proof of Theorem \ref{T:max} to show that $aS=0$. Since $R'$ is von Neumann regular modulo its nilradical, $S$ is regular. Thus, for each $P\in D(e)$, $S_P$ is flat and it is semisimple by Proposition \ref{P:regular}. Since $PR_P\ne 0$, it follows that $M_P=E_P$ for each $P\in D(e)$, and $M=E$. By \cite[Theorem 25.3]{AnFu92} $R'$ is Artinian. So, $R'$ is a finite product of local rings. We deduce that $e$ is a sum of orthogonal irreductible idempotents. So, \[\mathfrak{E}=\{e_P\mid P\in D(A)\ \mathrm{and}\ D(e_P)=\{P\}\}.\] 

$(2)\Rightarrow (5)$. Since $R$ is locally pure-semisimple by  Theorem \ref{T:locNoet} and coherent we do as above to define the pure ideal $A$. Then, by using Proposition \ref{P:eregular}, we show that each FP-injective $R'$-module is injective. So, $R'$ is Noetherian, and Artinian because each prime ideal is maximal. We end as above. 

By Proposition \ref{P:noabelien}, $(5)\Rightarrow (4), (2)$ and $(6)$ and by Proposition \ref{P:noabelien1}, $(5)\Rightarrow (7)$.
\end{proof}

\section{Baer's criterion}

The following two propositions are similar to Propositions \ref{P:noabelien} and \ref{P:noabelien1} and can be proven in the same way. They allow us to give non trivial examples of rings for which any flat-essential (P-flat-essential) extension of left modules is essential.

\begin{proposition}\label{P:main1}
Let $R$ be a  ring. Assume there exists a family $\mathfrak{E}$ of orthogonal central idempotents of $R$ satisfying the following conditions:
\begin{itemize}
\item[(a)] $R/R(1-e)$ is left perfect for each $e\in\mathfrak{E}$;
\item[(b)] $R/A$ is a von Neumann regular ring where $A=\oplus_{e\in\mathfrak{E}}Re$.
\end{itemize}
Then:
\begin{enumerate}
\item each flat-essential extension of left $R$-modules is essential;
\item any left $R$-module $M$ is cotorsion if and only if $\mathrm{Ext}_R^1(C,M)=0$ for each cyclic flat left $R$-module $C$;
\item $\mathcal{E}(M)/M$ is flat, FP-injective and regular for each left $R$-module $M$.
\end{enumerate}
\end{proposition}

We say that a ring $R$ is left {\bf strongly perfect} if each P-flat left $R$-module is projective. Clearly every left strongly perfect ring is perfect, but \cite[Proposition 4.8]{Cou11} shows that there exist Artinian commutative rings which are not strongly perfect. And \cite[Example 3.2]{BeCoSh14} is a strongly perfect ring by \cite[Theorem 4.11]{Cou11} and it is non-Artinian if $\Lambda$ is not finite.

\begin{proposition}\label{P:main2}
Let $R$ be a  ring. Assume there exists a family $\mathfrak{E}$ of orthogonal central idempotents of $R$ satisfying the following conditions:
\begin{itemize}
\item[(a)] $R/R(1-e)$ is left strongly perfect for each $e\in\mathfrak{E}$;
\item[(b)] $R/A$ is a von Neumann regular ring where $A=\oplus_{e\in\mathfrak{E}}Re$.
\end{itemize}
Then:
\begin{enumerate}
\item each P-flat-essential extension of left $R$-modules is essential;
\item any left $R$-module $M$ is Warfield cotorsion if and only if $\mathrm{Ext}_R^1(C,M)=0$ for each cyclic flat left $R$-module $C$;
\item $\mathcal{E}_W(M)/M$ is flat, FP-injective and regular for each left $R$-module $M$.
\end{enumerate}
\end{proposition}

Now we end by giving a description of commutative rings satisfying the Baer's criterion for (Warfield) cotorsion modules.

\begin{theorem}
\label{T:baer}
Let $R$ be commutative ring. Then the following conditions are equivalent:
\begin{enumerate}
\item $R_P$ is perfect for each maximal ideal $P$;
\item any $R$-module $M$ is cotorsion if and only if $\mathrm{Ext}_R^1(C,M)=0$ for each cyclic flat $R$-module $C$.
\end{enumerate}
\end{theorem}
\begin{proof}
$(2)\Rightarrow (1)$. Let $P$ be a maximal ideal and $M$ an $R_P$-module. If $C$ is a nonzero cyclic flat $R$-module, then $C_P$ is free over $R_P$. It follows that $\mathrm{Ext}^1_R(C,M)\cong\mathrm{Ext}^1_{R_P}(C_P,M)=0$. So, $M$ is cotorsion over $R$ and $R_P$. Since each $R_P$-module is cotorsion, $R_P$ is perfect.

$(1)\Rightarrow (2)$. Let $M$ be an $R$-module satisfying $\mathrm{Ext}^1_R(C,M)=0$ for any flat cyclic $R$-module $C$. Let $F$ be a free $R$-module, $K$ a pure submodule of $F$ and $\alpha:K\rightarrow M$ a homomorphism. We must prove that $\alpha$ extends to $F$. We consider the family $\mathcal{F}=\{(N,\beta)\}$ where $N$ is a pure submodule of $F$ containing $K$ and $\beta$ an extension of $\alpha$ to $N$. We consider the following order on $\mathcal{F}$: $(N,\beta)\leq (L,\gamma)$ if and only if $N\subseteq L$ and $\gamma\vert_N=\beta$. It is easy to see that we can apply Zorn Lemma to $\mathcal{F}$. So, let $(N,\beta)$ be a maximal element of $\mathcal{F}$. By way of contradiction suppose that $N\ne F$. Let $G=F/N$. There exists a maximal ideal $P$ such that $G_P\ne 0$. Since $R_P$ is perfect, $G_P$ is free over $R_P$. Thus there exists $x\in F\setminus N$ such that its image $y$ in $G_P$ verifies $(0:_{R_P}y)=0$. It follows that $(N:x)=0_P$ (see Proposition \ref{P:pireg}). Let $\delta:0_p\rightarrow M$ be the homomorphism defined by $\delta(a)=\beta(ax)$ for any $a\in 0_P$. Then $\delta$ extends to $R$. Now, let $\phi:N+Rx\rightarrow M$ be the homomorphism defined by $\phi(n+rx)=\beta(n)+\delta(r)$ for any $n\in N$ and $r\in R$. It is easy to check that $\phi$ is well defined. Let $H=N+Rx/N$. Then $H\cong R_P$. So, $H_P$ is a direct summand of $G_P$ and if $P'$ is another maximal ideal then $H_{P'}=0$. We successively deduce that $H$ is a pure submodule of $G$, $F/N+Rx$ is flat and $N+Rx$ is a pure submodule of $F$. This contradicts the maximality of $(N,\beta)$. Hence $N=F$ and $M$ is cotorsion. 
\end{proof}

It is easy to check that each P-flat cyclic left module is flat.

\begin{corollary}
\label{C:baer}
Let $R$ be commutative ring. Then the following conditions are equivalent:
\begin{enumerate}
\item $R_P$ is strongly perfect for each maximal ideal $P$;
\item any $R$-module $M$ is Warfield cotorsion if and only if $\mathrm{Ext}_R^1(C,M)=0$ for each cyclic flat $R$-module $C$.
\end{enumerate}
\end{corollary}
\begin{proof}
Let $G$ be a P-flat $R$-module. For each maximal ideal $P$ $G_P$ is P-flat. Since $R_P$ is strongly perfect, $G_P$ is free. Hence $G$ is flat. So, each cotorsion $R$-module is Warfield cotorsion.
\end{proof}

\section*{Acknowledgements}

This work was presented at the "second International Conference on Mathematics and Statistics (AUS-ICMS15)" held at American University of Sherjah, April 2-5, 2015. I thank again the organizers of this conference. I thank too the laboratory of mathematics Nicolas Oresme of the university of Caen Basse-Normandie which allowed me to participate to this conference.



\end{document}